\documentclass[a4paper]{article}

\title{Graphs with given automorphism group and large clique number}
\author{John Haslegrave}

\usepackage{amsmath,amsthm,fullpage}
\newcommand{\aut}{\operatorname{Aut}}

\newtheorem{theorem}{Theorem}
\newtheorem{proposition}[theorem]{Proposition}
\newtheorem{corollary}[theorem]{Corollary}

\begin{document}
	\maketitle
	
	\begin{abstract}Barbieri recently showed that the finite graphs realising any given finite automorphism group have unbounded genus, answering a question of Cornwell et al. In this note we give a short proof of a stronger result: they have unbounded clique number.
	\end{abstract}

In a paper concerned primarily with the minimum genus of graphs realising a given automorphism group, Cornwell, Doring, Lauderdale, Morgan and Storr asked how to determine the maximum genus of such a graph, or more precisely the maximum genus of a surface on which such a graph has a cellular embedding \cite[Open Question 3]{CDLMS}. Barbieri recently showed that the genus of a connected graph with given automorphism group is unbounded, and so this maximum does not exist for any group \cite{Barb}. Barbieri's construction has high genus by virtue of containing a high-dimensional hypercube. However, we note that this structure does not guarantee that other parameters of potential interest, such as the chromatic number, are also large. We give a simple proof that graphs realising any given automorphism group have unbounded clique number, and hence unbounded genus and chromatic number.
\begin{proposition}\label{prop}Let $ G $ be a non-complete finite graph, and let $ G '$ be the graph obtained from $ G $ by adding a disjoint clique with the same number of vertices and a perfect matching between the two. Then $\aut( G ')\cong \aut( G )$.
\end{proposition}
\begin{proof}
	Let $n=|V( G )|$. The statement holds for the unique non-complete graph on at most $2$ vertices, so we may assume $n\geq 3$. We partition $V( G ')$ as $V_1\cup V_2$, where $|V_1|=|V_2|=n$ and $ G '[V_1]\cong  G $ and $ G '[V_2]\cong K_n$. Furthermore, we may write $V_2=\{v'\mid v\in V_1\}$ such that $vv'\in E( G ')$ for each $v\in V_1$. Note that $ G '[V_2]$ is the unique copy of $K_n$ in $ G '$, and so any automorphism $\phi$ of $ G '$ must have $\phi(V_2)=V_2$ and therefore $\phi(V_1)=V_1$. Consequently, $\phi$ restricted to $V_1$ is an automorphism of $ G $. Conversely, for any $\psi\in\aut( G )$, there is a unique $\psi'\in\aut( G ')$ whose restriction to $V_1$ is $\psi$, obtained by setting $\psi'(v')=\psi(v)'$ for each $v\in V_1$. For any $\psi_1,\psi_2\in\aut( G )$ we have $(\psi_1\psi_2)'=\psi_1'\psi_2'$, since these agree on $V_1$, and hence $\aut( G ')\cong\aut( G ')$.
\end{proof}
\begin{corollary}For any graph $ G $ there exist connected graphs with automorphism group $\aut( G )$ and arbitrarily large clique number.
\end{corollary}	
\begin{proof}If $ G $ is not complete, this follows immediately by iterating the construction of Proposition \ref{prop}. If $ G $ is complete then $\aut( G )\cong S_n$ for some $n$, and we may apply the same argument to some non-complete graph with the same automorphism group: $\overline K_n$ if $n\geq 2$ or the graph obtained from the $6$-vertex path by adding an edge between the second and fourth vertices if $n=1$.
\end{proof}
Since the possible values of $\aut(G)$ range over all finite groups \cite{Frucht}, it follows that any finite group $\Gamma$ is the automorphism group of graphs of unbounded clique number.

\end{document}